\documentclass[12pt]{amsart}

\usepackage[T1]{fontenc}
\usepackage{lmodern}
\usepackage{amsmath, amssymb, amsthm}
\usepackage{geometry}
\usepackage{hyperref}
\usepackage{enumerate}
\usepackage{todonotes}
\usepackage{adjustbox}
\usepackage[normalem]{ulem}

\usepackage{booktabs}
\usepackage{rotating}  
\newtheorem{theorem}{Theorem}[section]
\newtheorem{prop}[theorem]{Proposition}
\newtheorem{lemma}[theorem]{Lemma}
\newtheorem{definition}[theorem]{Definition}
\newtheorem*{remark}{Remark}

\newcommand{\abs}[1]{\left\vert#1\right\vert}

\begin{document}

\title[Three-Term $b$-Concatenations in Pell-Type Sequences]{On the Diophantine Equations Arising from Three-Term $b$-Concatenations in Pell-Type Sequences}

\author[K. N. Ad\'edji]{Kou\`essi Norbert Ad\'edji}
\address{Institut de Math\'ematiques et de Sciences Physiques,
Universit\'e d'Abomey-Calavi, Abomey-Calavi, Benin}
\email{adedjnorb1988@gmail.com}

\author[M. Bliznac Trebje\v{s}anin]{Marija {Bliznac Trebje\v{s}anin}}
\address{University of Split, Faculty of Science, Ru\dj{}era Bo\v{s}kovi\'{c}a 33,
21000 Split, Croatia}
\email{marbli@pmfst.hr}

\author[J. Ple\v{s}tina]{Jelena {Ple\v{s}tina}*}
\thanks{*Corresponding author.}
\address{University of Split, Faculty of Science, Ru\dj{}era Bo\v{s}kovi\'{c}a 33,
21000 Split, Croatia}
\email{jplestina@pmfst.hr}

\begin{abstract}
We study the Diophantine problem of representing a Pell or Pell--Lucas number as the base-$b$ concatenation of three other Pell or Pell--Lucas numbers. Under the condition that the index of the middle block does not exceed that of the leading block, we obtain an explicit upper bound for the sequence indices in terms of $b$. For $2\leq b\leq15$, we determine exactly $51$ solutions up to the identification $Q_0=Q_1=2$, while there are $74$ solutions when distinct index choices are counted separately.
\end{abstract}

\maketitle 

\noindent{\it 2020 {Mathematics Subject Classification:}} 11B39, 11J68, 11J86

\noindent{\it Keywords}: {Pell numbers, Pell--Lucas numbers, $b$-concatenations, linear forms in logarithms, reduction method.}

\section{Introduction and Sequence Properties}

Let the Pell sequence $P=(P_{n})_{n\ge0}$ and the Pell--Lucas sequence $Q=(Q_{n})_{n\ge0}$ be defined by the initial values $P_0 = 0, P_1 = 1$, and $Q_0 = 2, Q_1 = 2$, along with the common binary linear recurrence relation for $n \ge 2$:
\begin{equation} \label{eq:common_recurrence}
U_n = 2U_{n-1} + U_{n-2}, \quad \text{where } U \in \{P, Q\}.
\end{equation}
The characteristic equation is $x^2 - 2x - 1 = 0$, with roots $\alpha = 1 + \sqrt{2}$ and $\beta = 1 - \sqrt{2}$. Note that $\alpha \beta = -1$, and hence $\beta = -1/\alpha$. The absolute value of the conjugate root is $|\beta| = \sqrt{2} - 1 < 1$.

The Binet formulas are
\begin{equation*}
P_n = \frac{\alpha^n - \beta^n}{2\sqrt{2}} = c(\alpha^n - \beta^n), \quad Q_n = \alpha^n + \beta^n,
\end{equation*}
where $c = \frac{1}{2\sqrt{2}}$.

Let $\lambda_P=c$ and $\lambda_Q=1$ denote the coefficients of $\alpha^j$, and let $\varepsilon_P(j)=-c\beta^j$ and $\varepsilon_Q(j)=\beta^j$ denote the second summands in the Binet formulas. Then we can write
\begin{equation}\label{eq:binet_general}
    U_j=\lambda_U\alpha^j+\varepsilon_U(j).
\end{equation}
Since $|\beta|=\alpha^{-1}$ and $0<c<1$, we have $|\varepsilon_U(j)|\leq \alpha^{-j}, \;  j\geq 0.$

The following bounds hold for all $n \ge 1$:
\begin{equation} \label{eq:bounds_P_Q}
\alpha^{n-2} \le P_n \le \alpha^{n-1}, \quad \alpha^{n-1} \le Q_n < \alpha^{n+1}.
\end{equation}
Moreover, for every $n \geq 1$, we have 
\begin{equation}
P_n < Q_n < 4P_n.
\end{equation}
These inequalities follow from \eqref{eq:common_recurrence} by induction.

Let $b \ge 2$ be an integer and let $l_b(n)=\lfloor \log_b n \rfloor + 1$ denote the number of digits of a positive integer $n$ in base $b$. We also define $l_b(0)=1$. 

\begin{definition}
Let $b\geq 2$ be an integer. A positive integer $N$ is called a base-$b$ concatenation (or $b$-concatenation) of nonnegative integers $a_1$, $a_2$, and $a_3$ if it can be expressed as 
$$N=a_1 b^{d+l}+a_2 b^l+a_3,$$
where $d$ and $l$ are the numbers of digits of $a_2$ and $a_3$, respectively, in base $b$. 
\end{definition}

We investigate when a Pell or Pell--Lucas number is a $b$-con\-ca\-te\-na\-ti\-on of three Pell or Pell--Lucas numbers. This leads to the study of 16 different equations, which can be expressed in the general form:
\begin{equation}\label{eq:main}
   X_k=b^{d+l}Y_m+b^lZ_n+T_r,
\end{equation}
where $X, Y, Z, T \in \{P, Q\}$, $d = l_b(Z_n)$ and $l = l_b(T_r)$, and $k, m, n, r$ are nonnegative integers denoting the indices of the sequences. Furthermore, we require $m\geq 1$. Specifically, if $Y_m=P_m$, we exclude the case $m=0$ because $P_0=0$, and concatenating a leading $0$  reduces the expression to a concatenation of fewer than three numbers. Additionally, if $Y_m=Q_m$, we can assume $m\geq 1$ without loss of generality since $Q_0=Q_1=2$. 

Concatenations of an arbitrary number of terms from the same binary recurrent sequence were studied by Banks and Luca \cite{BanksLuca2005}. They proved an ineffective finiteness result for the general case and completely resolved the case of the concatenation of two elements of the Fibonacci sequence. Meanwhile, mixed Pell and Pell--Lucas concatenations of length two were investigated in \cite{AdedjiFayeTogbe2024} and \cite{abt}, while three-term concatenations have been considered for the Padovan and Perrin sequences in~\cite{Erduvan2026}. In the present work, we continue the study of mixed Pell and Pell--Lucas concatenations by considering concatenations of length three.

Our analysis is restricted to the condition $n \le m$, even though solutions with $n>m$ do exist, for example, $P_5=2^4\cdot P_1+2^1\cdot Q_2+P_1$ or $Q_{13} = 9^4 \cdot Q_3 + 9^2 \cdot Q_4 + Q_4$. This assumption is crucial for applying lower bounds for linear forms in logarithms. Without this condition, the logarithmic height of the algebraic number in the second application of Matveev's theorem would grow proportionally with $n$, leading to bounds that fail to restrict the index $k$ to a finite range. Under this assumption, we obtain an upper bound on $k$, as stated in our main theorem.

\begin{theorem}\label{tm:main}
Let $b \ge 2$ be an integer and let $X,Y,Z,T\in\{P,Q\}$. Then, the Diophantine equation
\begin{equation*}
X_k = b^{d+l}Y_m + b^lZ_n + T_r,
\end{equation*}
where $d = l_b(Z_n)$, and $l = l_b(T_r)$, has only finitely many solutions in nonnegative integers $k, m, n, r$ under the conditions $m \ge 1$ and $n \le m$. 
Namely, we have
$$k < 8 \cdot 10^{30} \cdot \log^3 b.$$
\end{theorem}
\smallskip

Theorem~\ref{tm:main} establishes the finiteness of the solutions for all sixteen equations \eqref{eq:main}. We now turn to the explicit determination of the solutions for a finite range of bases. Further analysis and exhaustive computation for $2\le b\le 15$ lead to the following result.

\begin{theorem}\label{thm2:main}   As $b$ ranges over the integers with $2\leq b\leq15$, the sixteen Diophantine equations
\eqref{eq:main}, where $X,Y,Z,T\in\{P,Q\}$, have exactly $51$
solutions satisfying
$m\geq1,$ $ 0\leq n\leq m,$
up to the identification $Q_0=Q_1=2$. If occurrences of $Q_0$ and
$Q_1$ are distinguished by their indices, the corresponding number
of solutions is $74$. The $51$ solutions are distributed by base as follows:
\begin{itemize}
    \item $b=2$: $10$ solutions, with a maximum index of $k=8$,
    \item $b=3$: $3$ solutions, with a maximum index of $k=4$,
    \item $b=4$: $12$ solutions, with a maximum index of $k=10$,
    \item $b=6$: $1$ solution, given by $Q_9 = 6^{1+2}\cdot P_4 + 6^2\cdot P_3 + Q_3$,
    \item $b=8$: $3$ solutions, with a maximum index of $k=6$,
    \item $b=9$: $5$ solutions, with a maximum index of $k=8$,
    \item $b=12$: $1$ solution, given by $P_7 = 12^{1+1}\cdot P_1 + 12^1\cdot Q_1 + P_1$,
    \item $b=13$: $5$ solutions, with a maximum index of $k=14$,
    \item $b=14$: $10$ solutions, with a maximum index of $k=10$,
    \item $b=15$: $1$ solution, given by $Q_8 = 15^{1+1}\cdot P_3 + 15^1\cdot P_1 + Q_3$.
\end{itemize}
For bases $b \in \{5, 7, 10, 11\}$, there are no solutions satisfying $n \le m$. 
\end{theorem}

\begin{remark}
While $P_2 = 2$ also coincides with the values of $Q_0$ and $Q_1$, occurrences of $P_2$ are counted as distinct solutions because they correspond to a different sequence and thus a different equation. 
\end{remark}

The organization of the paper is as follows. In Section~\ref{sec2}, we briefly gather the necessary auxiliary tools. In Section~\ref{sec3}, we examine equation \eqref{eq:main} by deriving several basic inequalities between the parameters and resolving the initial small cases. We then analyze two linear forms in logarithms in Section~\ref{sec4} to obtain explicit upper bounds for the index $k$. To fully resolve the equation and prove Theorem~\ref{thm2:main}, Section~\ref{sec5} is dedicated to reducing these large bounds to a computationally feasible range. This allows us to complete an exhaustive computer search for the given bases. Note that our general approach relies on the Baker--Davenport reduction method, utilizing standard auxiliary tools such as Matveev's lower bound and Legendre's criterion.

\section{Preliminaries}\label{sec2}

 In this section, we recall some fundamental results concerning lower bounds for nonzero linear forms in logarithms of algebraic numbers, alongside other auxiliary lemmas required for our proofs.
\begin{definition}Let $\eta$ be an algebraic number of degree $t,$ let $a_0 \ne 0$ be the leading coefficient of its minimal polynomial over $\mathbb{Z}$ and let $\eta=\eta^{(1)},\ldots,\eta^{(t)}$ denote its conjugates. The logarithmic height of $\eta$ is defined by
$$
 h(\eta)= \frac{1}{t}\left(\log |a_0|+\sum_{j=1}^{t}\log\max\left(1,\left|\eta^{(j)} \right| \right) \right).
$$\end{definition}
If $p$ and $q$ are integers such that $q>0$ and $\gcd (p, q)=1,$ then for $\eta=p/q$ the above definition reduces to $h(\eta)=\log(\max\{|p|,q\}).$ 

We recall some properties of the logarithmic height (Property 3.3 of \cite{Wald2000}): \\
If $\eta_1$ and $\eta_2$ are algebraic numbers, then we have
\begin{align*}
 h(\eta_1\eta_2) &\leq h(\eta_1) + h(\eta_2), \\
 h(\eta_1 \pm \eta_2) &\leq h(\eta_1)+ h(\eta_2) +\log2.
\end{align*}
If $\eta_1 \neq 0$ is a nonzero algebraic number and $j\in \mathbb{Z}$, then we have
\begin{align*}
h(\eta_1^j)&=|j|h(\eta_1).
\end{align*}

To obtain lower bounds on linear forms in logarithms, we use the following fundamental result by Matveev \cite{Matveev2000}, as formulated in \cite{BMS:2006}.
\begin{lemma}[Theorem~9.4 of \cite{BMS:2006}]\label{tm:BMS}
Let $\gamma_1,\dots,\gamma_s$ be positive real algebraic numbers and let $b_1,\dots,b_s$ be nonzero integers. Let $D$ be the degree of the number field $\mathbb{Q}(\gamma_1,\dots,\gamma_s)$ over $\mathbb{Q}$ and let $A_j$ be a positive real number satisfying  
$$A_j\geq\max\{Dh(\gamma_j),|\log\gamma_j|,0.16\},\quad \textrm{for }j=1,\dots,s. $$ 
Assume that 
$$B\geq \max\{|b_1|,\dots,|b_s|\}.$$
If $\Lambda:=\gamma_1^{b_1}\cdots\gamma_s^{b_s}-1\neq 0$, then
\begin{equation*}
|\Lambda|\geq \exp(-1.4\cdot30^{s+3}\cdot s^{4.5}\cdot D^2(1+\log D)(1+\log B)A_1\cdots A_s).
\end{equation*}
\end{lemma}
The next auxiliary result from \cite{SGL} will be useful for transforming inequalities.
\begin{lemma}[Lemma~7 of \cite{SGL}]\label{lemma:supporting}
If $ \ell \geq 1$, $H>(4 \ell^2)^\ell$ and $H>L/(\log L)^\ell$, then
$$L<2^\ell H(\log H)^\ell.$$
\end{lemma}

As one step of the proof for the application, we will apply the reduction method originally introduced by Baker and Davenport \cite{bd}. The following is a variation of the result of Dujella and Peth\H{o} (see
\cite[Lemma 5]{Dujella-Peto}), with modifications in the first part of the lemma by Bravo, Gómez and Luca (see \cite[Lemma~1]{BGF}).

For a real number $x,$ we write $\left\Vert x\right\Vert$ for the distance from $x$ to the nearest integer.
\begin{lemma}\label{lemma:reduction}
	Let $M$  be a positive integer, let $p/q$ be a convergent of the continued fraction of the irrational $\tau$ such that $q> 6M$, and let $A,B,\mu$ be some real numbers with $A> 0$ and $B> 1$. Let
	$$
	\varepsilon=||\mu q||-M\cdot||\tau q||.
	$$
	If    $\varepsilon>0$, then there is no solution of the inequality
	\begin{align}\label{eqn:baker_d}
	0 <\abs{m\tau-n+\mu}<AB^{-w},
	\end{align}
	in positive integers $m,n$ and $w$ with 
	$$
	m\leq M \quad \text{and} \quad w\geq\dfrac{\log(Aq/\varepsilon)}{\log B}.
	$$ 
\end{lemma}

When Lemma~\ref{lemma:reduction} cannot be applied because the corresponding
value of $\varepsilon$ is non-positive, we will use the following criterion of
Legendre after reducing the resulting linear form to a homogeneous one.
\begin{lemma}[Lemma~1.7 of \cite{abt}]\label{lemma:legendre}
Let $\kappa$ be a real number and $x,y$ integers such that 
$$\left|\kappa-\frac{x}{y} \right|<\frac{1}{2y^2}.$$
Then $x/y=p_k/q_k$ is a convergent to $\kappa$. Furthermore, let $M$ and $N$ be nonnegative integers such that $q_N>M$. Put $a(M):=\max\{a_i: i=0,1,2,\dots,N\}$. Then the inequality
$$\left|\kappa-\frac{x}{y} \right|\geq \frac{1}{(a(M)+2)y^2}$$
holds for all pairs $(x,y)$ of positive integers with $0<y<M$.
\end{lemma}

\section{General Bounds for Variables}\label{sec3}
From equation \eqref{eq:main}, the lower bound $X_k\geq b^{d+l}Y_m \ge 4$ immediately implies that $k \ge 2$. For this minimal index $k=2$, the equation holds in exactly one trivial case.

\begin{prop}\label{prop:kje2}
  If $k=2$ in the equation \eqref{eq:main}, then
  $$b=2, \, X=Q, \, Y=Z=T=P, \, m=n=1, \, r=0, \, d=l=1$$
and in this case
$$
Q_2=2^{1+1}P_1+2^1P_1+P_0.
$$
\end{prop}
\begin{proof}
    Since $b^{d+l}Y_m\geq4$ and $P_2=2$, the case $X=P$ is impossible. Hence $X=Q$ and $X_2=Q_2=6$.
Moreover, since $b^{d+l}Y_m\geq4$ and the remaining terms are nonnegative, we must have $b=2, d+l=2, Y_m=1$. Thus $d=l=1, Y=P$ and $m=1$.
Since $d=1$ and $b=2$, the middle block $Z_n$ is either $0$ or $1$. Hence $Z=P$ and $n \in \{0,1\}$. 
Since $l=1$ and $b=2$, the last block $T_r$ is either $0$ or $1$. Hence $T=P$ and $r \in \{0,1\}$.
The equation becomes
$$ 
6=2^2P_1+2Z_n+T_r=4+2P_n+P_r.
$$
Therefore $2P_n+P_r=2$, which gives $P_n=1$ and $P_r=0$. Hence $Z=P, n=1, T=P$ and $r=0$.
\end{proof}

Since this trivial solution satisfies the statement of Theorem \ref{tm:main}, we can from now on assume $k\geq 3$.

\begin{lemma}\label{lem:k_numerical_lower} 
If $X=P$ in the equation \eqref{eq:main} then $k\geq 4$, except in the case  $b=2,$  $d=l=1$,    $r=m=1$, $n=0$, $T=Z=Y=P$, when $k=3$. 
Moreover, if $X=P$, $T=Q$ and $r\in\{0,1\}$ then $k\geq 5$.
\end{lemma}

\begin{proof}
In the equation \eqref{eq:main}, we have $m\geq 1$, $n,r\geq 0$, $b\geq 2$, $d+l\geq 2$. The right-hand side of the equation is at least $2^2\cdot 1+0+0=4$. Hence, $k\geq 3$. Moreover, if $k=3$, then $X_3=P_3=5$, implying $b=2$, $d=l=1$, $Z_n=0$ and $Y_m=T_r=1$ implying $Z=T=Y=P$ and $n=0$ and $r=m=1$.  

If $X=P$, $T=Q$, $r\in\{0,1\}$ and $k=4$, then $P_4=12$ and $T_r=Q_0=Q_1=2$. Since $12\geq b^{d+l}\geq b^2$, $b\in\{2,3\}$. 
If $b=2$, then $l=l_2(Q_0)=2$, and therefore $12=2^{d+2}Y_m+2^2Z_n+2$. This gives $10=2^{d+2}Y_m+4Z_n$, which is impossible, since the right-hand side is divisible by $4$. 
If $b = 3$, then $l=l_b(Q_0)=1$, and therefore $12=3^{d+1}Y_m+3Z_n+2$. Thus $10=3(3^dY_m+Z_n )$, which cannot hold. Hence $k\neq 4$.
Since in the case $k=3$, we have $T=P$, it follows that $k\geq5$ in the case $X=P$, $T=Q$, $r\in\{0,1\}$. 
\end{proof}



With the notation of \eqref{eq:main}, the following bounds hold. 
First, from \eqref{eq:bounds_P_Q}, for every $U\in\{P,Q\}$ and every $j\geq1$, we have
\begin{equation}\label{eq:seq_bounds}
    \alpha^{j-2}\leq U_j<\alpha^{j+1}.
\end{equation}
Moreover, from the definition of $l_b(N)$, for every $U\in\{P,Q\}$ and every $j\geq0$ such that $U_j\geq1$, we have
\begin{equation}\label{eq:digit_bounds}
    U_j<b^{l_b(U_j)}\leq bU_j.
\end{equation}

Since the value $P_0=0$ may occur in the second or third block of \eqref{eq:main}, we also need bounds that remain valid when $U_j=0$. Thus, for every $U\in\{P,Q\}$ and every $j\geq0$, we have
\begin{equation}\label{eq:seq_bounds_max}
    \max\{U_j,1\}<\alpha^{j+1},
\end{equation}
and, using the convention $l_b(0)=1$,
\begin{equation}\label{eq:digits}
    U_j<b^{l_b(U_j)}\leq b\max\{U_j,1\}.
\end{equation}

\begin{lemma} \label{lem:general_bounds}
Let $b\geq2$ and let $X,Y,Z,T\in\{P,Q\}$. Let $(k,m,n,r,d,l)$ be a solution of \eqref{eq:main}. Assume that $m \geq 1$ and $n\leq m$. Then
\begin{enumerate}[a)]
    \item $k>m$;
    \item $k>r$;
    \item If $X=P$, 
    then $k-r \geq 2$;
    \item $d+l<1.28(k+1)$;
    \item 
    $k -r < 2m+5+ 
    \frac{\log(b^2+b+1)}{\log \alpha}.$
\end{enumerate}
\end{lemma}
\begin{proof}
We prove a). From \eqref{eq:main}, we have 
$X_k \geq b^{d+l}Y_m \geq 4Y_m > Y_m.$ 
We distinguish three cases.
Assume first that $X=Y$. Since both sequences $(P_j)_{j\geq1}$ and
$(Q_j)_{j\geq1}$ are strictly increasing, from $X_k>Y_m=X_m$ we obtain $k>m$.
If $X=P$ and $Y=Q$, then $P_k=X_k>Y_m=Q_m$. Since $Q_m>P_m$ for all $m\geq1$, we get $P_k>P_m$. Since $m\geq1$ and the Pell sequence $(P_j)_{j\geq1}$ is strictly increasing, this implies $k>m$.
If $X=Q$ and $Y=P$, then $Q_k=X_k \geq 4Y_m=4P_m$. Since $Q_m<4P_m$ for all $m\geq1$, we get $Q_k>Q_m$. Since the Pell--Lucas sequence $(Q_j)_{j\geq1}$ is strictly increasing, it follows that $k>m$.

We prove b).
If $r=0$, then $k>r$ follows from a), since $k>m\geq1$. 
Assume now that $r\geq1$. From \eqref{eq:main}, we have $X_k>T_r$. 
If $X=T$, then $X_k>T_r=X_r$. Since the corresponding sequence is strictly increasing from the index $1$ on, and since $r\geq1$, it follows that $k>r$.
If $X=P$ and $T=Q$, then $P_k=X_k>T_r=Q_r$. Since $Q_r>P_r$ for $r\geq1$, we obtain $P_k>P_r$. Hence $k>r$. 
It remains to consider the case $X=Q$ and $T=P$. Since $r\geq1$ we have $l=l_b(P_r)$ and therefore $b^l>P_r$. Since $d\geq1$, $b\geq2$, and $Y_m\geq1$, we have $b^{d+l}Y_m\geq b^{l+1}>2P_r$. Using \eqref{eq:main}, we obtain $Q_k=X_k=b^{d+l}Y_m+b^lZ_n+P_r>2P_r+P_r=3P_r$. Using $Q_r=2P_r+2P_{r-1}$ and $2P_{r-1}\leq P_r$ for $r\geq1$, we get $Q_r\leq3P_r$. Thus 
$Q_k>3P_r\geq Q_r$, which implies $k>r$.

We prove c). If $X=P$ and $T=Q$, this implies $P_k=X_k>T_r=Q_r$. If $r=0$, then $Q_0=2=P_2$, and therefore $P_k>P_2$. Hence $k\geq3$, so $k-r\geq3$.
Let now $r\geq1$. Using the identity $Q_r=P_{r+1}+P_{r-1}$, we get $Q_r\geq P_{r+1}$. Thus $P_k>Q_r\geq P_{r+1}$, which implies $k>r+1$, so $k-r\geq2$.

If $T=P$ and $r=0$, then the exceptional case in Lemma \ref{lem:k_numerical_lower} does not occur, and hence $k\geq4$. Therefore $k-r=k>2$.
If $T=P$ and $r\geq1$, we use the inequality \eqref{eq:digit_bounds} to obtain
$P_k=b^{d+l}Y_m+b^lZ_n+P_r>P_r(b^dY_m+Z_n+1)\geq 3P_r$. Since the sequence $P$ is increasing we have $k>r$. If $k-r=1$, we would get 
$P_{r+1}=2P_r+P_{r-1}>3P_r$
implying $P_{r-1}>P_r$, a contradiction since the sequence $(P_k)_{k\geq0}$ is strictly increasing. Therefore $k-r\geq2$.

We prove d). Since $Y_m\geq1$, from \eqref{eq:main} we get $b^{d+l} \leq X_k$. By \eqref{eq:seq_bounds}, we have $X_k<\alpha^{k+1}$. Hence $b^{d+l}<\alpha^{k+1}$. Taking logarithms gives $d+l<(k+1)\frac{\log\alpha}{\log b}$. Since $b\geq2$ and $\log\alpha/\log2<1.28$, it follows that $d+l<1.28(k+1)$.

We prove e). By \eqref{eq:digits}, we have $b^d\leq b\max\{Z_n,1\}$ and $b^l\leq b\max\{T_r,1\}$. Hence, from \eqref{eq:main}, we get
$$
\begin{aligned}
X_k
&=b^{d+l}Y_m+b^lZ_n+T_r \\
&\leq b^2Y_m\max\{Z_n,1\}\max\{T_r,1\}
+b\max\{Z_n,1\}\max\{T_r,1\}
+\max\{T_r,1\}.
\end{aligned}
$$
Since $Y_m\geq1$ and $\max\{Z_n,1\}\geq1$, it follows that
$$X_k\leq (b^2+b+1)Y_m\max\{Z_n,1\}\max\{T_r,1\}.$$ Now, using \eqref{eq:seq_bounds_max} and the assumption $n\leq m$, we obtain $$
Y_m<\alpha^{m+1},
\;
\max\{Z_n,1\}<\alpha^{m+1},
\;
\max\{T_r,1\}<\alpha^{r+1}.$$
Therefore $X_k<(b^2+b+1)\alpha^{2m+r+3}$.
On the other hand, by \eqref{eq:seq_bounds}, we have $\alpha^{k-2}\leq X_k$. Combining the last two inequalities gives
$\alpha^{k-2}<(b^2+b+1)\alpha^{2m+r+3}$,
and taking logarithms we obtain
$ k-r<2m+5+\frac{\log(b^2+b+1)}{\log\alpha} $.
\end{proof}



\section{Proof of Theorem \ref{tm:main}}\label{sec4}

We consider all 16 equations under the hypothesis $n \le m$. The derivations follow a parallel structure: we first obtain a bound for $m$ from a first linear form and then use a second linear form to derive an absolute bound for $k$.

    \subsection{The First Linear Form}
Substituting \eqref{eq:binet_general} for $X_k$ and $Y_m$ into \eqref{eq:main}, we obtain 
$$\lambda_X\alpha^k+\varepsilon_X(k)
=
b^{d+l}\bigl(\lambda_Y\alpha^m+\varepsilon_Y(m)\bigr)
+b^lZ_n+T_r.$$
After rearranging,
\begin{equation*}
\lambda_X\alpha^k-b^{d+l}\lambda_Y\alpha^m
=
b^{d+l}\varepsilon_Y(m)+b^lZ_n+T_r-\varepsilon_X(k).
\end{equation*}
Dividing both sides by $b^{d+l}\lambda_Y\alpha^m$, we obtain
\begin{equation}\label{eq:first_rearr1}
\frac{\lambda_X}{\lambda_Y}\alpha^{k-m}b^{-(d+l)}-1
=
\frac{
b^{d+l}\varepsilon_Y(m)+b^lZ_n+T_r-\varepsilon_X(k)}
{b^{d+l}\lambda_Y\alpha^m}.
\end{equation}
We define
\begin{equation}\label{eq:Gamma1_def}
  \Gamma_1:=\frac{\lambda_X}{\lambda_Y}\alpha^{k-m}b^{-(d+l)}-1.
\end{equation}
From \eqref{eq:first_rearr1} and \eqref{eq:Gamma1_def} we have
\begin{equation}\label{eq:Gamma1_char}
  \Gamma_1=\frac{
b^{d+l}\varepsilon_Y(m)+b^lZ_n+T_r-\varepsilon_X(k)}
{b^{d+l}\lambda_Y\alpha^m}.
\end{equation}

Now we will show that for every solution of \eqref{eq:main} satisfying  $m\geq1$, we have
\begin{equation}\label{lem:Gamma1_bound}
    |\Gamma_1|<8\alpha^{-m}.
\end{equation}
Using \eqref{eq:Gamma1_char} and $|\varepsilon_U(j)|\leq\alpha^{-j}$, we get
\begin{equation*}
|\Gamma_1|
\leq
\frac{b^{d+l}|\varepsilon_Y(m)|}{b^{d+l}\lambda_Y\alpha^m}
+
\frac{b^lZ_n}{b^{d+l}\lambda_Y\alpha^m}
+
\frac{T_r}{b^{d+l}\lambda_Y\alpha^m}
+
\frac{|\varepsilon_X(k)|}{b^{d+l}\lambda_Y\alpha^m},
\end{equation*}
and since $\lambda_Y\geq 1/(2\sqrt2)$, we have 
\begin{equation}\label{eq:Gamma1_bound1}
|\Gamma_1|
\leq
\frac{1}{\lambda_Y}
\left(
\alpha^{-2m}
+
\frac{Z_n}{b^d\alpha^m}
+
\frac{T_r}{b^{d+l}\alpha^m}
+
\frac{\alpha^{-k-m}}{b^{d+l}}
\right).
\end{equation}
By \eqref{eq:digits}, we have $Z_n<b^d$ and $T_r<b^l$. Therefore, from \eqref{eq:Gamma1_bound1} it follows
\begin{equation}\label{eq:Gamma1_bound2}
|\Gamma_1|
<
2\sqrt2
\left(
\alpha^{-2m}
+
\alpha^{-m}
+
\frac{\alpha^{-m}}{b^d}
+
\frac{\alpha^{-k-m}}{b^{d+l}}
\right).
\end{equation}
By Lemma \ref{lem:general_bounds}, we have $k>m$, and since $d,l\geq1$ and $b\geq2$, from \eqref{eq:Gamma1_bound2}, we obtain
\begin{equation*}
|\Gamma_1|
<
2\sqrt2
\left(
1+1+\frac12+\frac14
\right)\alpha^{-m}
<
8\alpha^{-m}.
\end{equation*}

Now we will show that for every solution of \eqref{eq:main} satisfying  $m\geq1$, we have
    \begin{equation}\label{lem:Gamma1_nonzero}
        \Gamma_1 \neq 0.
    \end{equation}
Assume, to the contrary, that $\Gamma_1 = 0$.
    Then 
    \begin{equation*}
        \frac{\lambda_X}{\lambda_Y}\alpha^{k-m}=b^{d+l}.
    \end{equation*}
Taking norm $N$ on $\mathbb Q(\sqrt2)$ and using $N(\alpha)=-1$, we obtain
    \begin{equation}\label{eq:nonzero}
        N\left(\frac{\lambda_X}{\lambda_Y}\right)(-1)^{k-m}=b^{2(d+l)}.
    \end{equation}
However, 
\begin{equation*}
        \frac{\lambda_X}{\lambda_Y} \in \{1,c,c^{-1}\},
    \end{equation*}
and therefore 
\begin{equation*}
        N\left(\frac{\lambda_X}{\lambda_Y}\right)(-1)^{k-m} \in \{\pm1,\pm1/8,\pm8\}.
    \end{equation*}
On the other hand, since $b\geq2$ and $d+l\geq2$, the right-hand side of \eqref{eq:nonzero} satisfies
$b^{2(d+l)} \geq 2^4.$
This is impossible. Therefore $\Gamma_1\neq0$.

\subsubsection{Matveev Parameters for $\Gamma_1$}
We apply Lemma \ref{tm:BMS} to 
$$
\Gamma_1=\frac{\lambda_X}{\lambda_Y}\alpha^{k-m}b^{-(d+l)}-1.
$$
We define
$$
\gamma_1=\alpha,
\qquad
\gamma_2=b,
\qquad
\gamma_3=\frac{\lambda_X}{\lambda_Y}
$$
and 
$$
b_1=k-m,
\qquad
b_2=-(d+l),
\qquad
b_3=1.
$$
$\gamma_1, \gamma_2, \gamma_3 \in \mathbb Q(\sqrt2)$, so $D=2$.
Since
$$
h(\alpha)=\frac12\log\alpha,
\qquad
h(b)=\log b,
$$
we may take 
$$A_1=\log\alpha,
\qquad
A_2=2\log b.$$
Since $\frac{\lambda_X}{\lambda_Y} \in \{1,c,c^{-1}\}$, we may take $A_3=\log 8$.
By Lemma \ref{lem:general_bounds}, we have $d+l<1.28(k+1)$ and for $k\geq2$, this gives $d+l<3k$.
We have
$$\max\{|k-m|,|-(d+l)|, |1| \}  < \max\{k,3k,1 \}=3k,$$
thus we may take $B=3k$.
Because of \eqref{lem:Gamma1_nonzero} and by Lemma \ref{tm:BMS} we have
\begin{equation}\label{Gamma1_bound1}
    \log|\Gamma_1|
>
-C_1(1+\log(3k))(\log\alpha)(2\log b)(\log8),
\end{equation}
where
\begin{equation*}
C_1=1.4\cdot30^6\cdot3^{4.5}\cdot2^2(1+\log2).
\end{equation*}
On the other hand, from inequality \eqref{lem:Gamma1_bound}, we have
\begin{equation}\label{Gamma1_bound2}
\log|\Gamma_1|<\log8-m\log\alpha.
\end{equation}
Combining \eqref{Gamma1_bound1} and \eqref{Gamma1_bound2}, we get
\begin{equation*}
m\log\alpha-\log8
<
C_1(1+\log(3k))(\log\alpha)(2\log b)(\log8) 
\end{equation*}
which gives
\begin{equation}\label{Gamma1_bound_m}
m
<4.04\cdot10^{12}(\log b)(1+\log(3k)).
\end{equation}

\subsection{The Second Linear Form $\Gamma_2$}

Rearranging the initial equation \eqref{eq:main}, after using \eqref{eq:binet_general} for $X_k$ and $T_r$, factoring out $\lambda_X\alpha^k$, and dividing the entire equation by the corresponding dominant factor yields our second linear form
\begin{equation*}
    \Gamma_{2} :=1- \alpha^{-k} b^l
    \frac{b^d Y_m+Z_n}{\lambda_X(1-\lambda_T\lambda_X^{-1}\alpha^{r-k})} 
    = \frac{\varepsilon_T(r)-\varepsilon_X(k)}{\lambda_X\alpha^k(1-\lambda_T\lambda_X^{-1}\alpha^{r-k})}.
\end{equation*}
Note that we have $1-\lambda_T\lambda_X^{-1}\alpha^{r-k} > 0$ for each choice of $\lambda_T$ and $\lambda_X$. This follows from Lemma \ref{lem:general_bounds}, which ensures that when $\lambda_T\lambda_X^{-1} > 1$, specifically, when $X=P$ and $T=Q$, it holds $k-r \ge 2$, maintaining $\lambda_T\lambda_X^{-1}\alpha^{r-k} < 1$. For all other choices, the coefficient ratio is at most $1$ and $k>r$, again ensuring that  $\lambda_T\lambda_X^{-1}\alpha^{r-k}<1$.

Now we will show that 
\begin{equation*}
    |\Gamma_{2}| \leq \frac{3}{\alpha^k}.
\end{equation*}
To prove this inequality, notice that since $|\varepsilon_U(j)|\leq \alpha^{-j}$, we have $|\Gamma_2|\leq \alpha^{-k} E$, where  
\begin{equation*}
    0<E=\frac{\lambda_X^{-1}\alpha^{-r}(\alpha^{-(k-r)}+1)}{1-\lambda_T\lambda_X^{-1}\alpha^{-(k-r)}}.
\end{equation*} 
Hence, it suffices to show that $E\leq 3$.
First, consider the case when $X=Q$. Then $\lambda_X=1$ and $\lambda_T\lambda_X^{-1}\leq 1$. Since from Lemma \ref{lem:general_bounds} it holds $k-r\geq 1$, we have $E\leq \alpha <3$. Second, consider the case when $X=T=P$. Then $\lambda_X=\lambda_T=c$. If $r\geq 1$ we have, by Lemma \ref{lem:general_bounds}, $k-r\geq 2$, then
$$ E\leq \frac{2\sqrt{2}}{\alpha}\frac{(\alpha^{-2}+1)}{1-\alpha^{-2}}<3.$$
If $r=0$ then $k\geq 4$ by Lemma \ref{lem:k_numerical_lower} and we have $E\leq3$. 
Last, if $X=P$ and $T=Q$, then $\lambda_X=c$ and $\lambda_T\lambda_X^{-1}=2\sqrt{2}$, and $k-r\geq 2$. If $r\geq 1$, we get $E<3$. If $r=0$, by Lemma \ref{lem:k_numerical_lower} $k\geq 5$ hence again $E<3$. 

Also, $\Gamma_2\neq 0$. If $\Gamma_{2}=0$, then $\varepsilon_T(r)=\varepsilon_X(k)$. If $T=X$, we would have $\beta^k=\beta^r$, implying $k=r$, which cannot hold.  If $T\neq X$, then $\beta^{k-r}=-c$ or $\beta^{k-r}=-1/c$. The norm of the left-hand side is equal to $\pm 1$, while the norms of the right-hand side are $-1/8$ and $-8$, respectively, implying a contradiction in both cases.

\subsubsection{Matveev Parameters for $\Gamma_2$}
We apply Lemma \ref{tm:BMS} again, this time  $s=3$. The degree remains $D=2$. Furthermore,
\begin{itemize}
    \item $\gamma_1 = \alpha, \quad b_1 = -k \quad \implies A_1 = \log \alpha,$
    \item $\gamma_2 = b, \quad b_2 = l \quad \implies A_2 = 2\log b,$
    \item $\gamma_3 = \lambda_X^{-1}\frac{b^d Y_m + Z_n}{1-\lambda_T\lambda_X^{-1}\alpha^{r-k}}, \quad b_3 = 1.$

\end{itemize}

 By Lemma \ref{lem:general_bounds}, we have $l<d+l<3k$, therefore we can take $B=3k$.

\noindent\textbf{Logarithmic Height $h(\gamma_3)$.}
We have
$$
h(\gamma_{3})
\leq
h(b^d Y_m+Z_n)+h(\lambda_X(1-\lambda_T\lambda_X^{-1}\alpha^{r-k})).
$$
For the numerator, we use \eqref{eq:digits} and \eqref{eq:seq_bounds_max}. Since $n\leq m$, we have
$
b^d\leq b\max\{Z_n,1\},$
$Y_m<\alpha^{m+1},$
$\max\{Z_n,1\}<\alpha^{m+1}.$
This gives:
\begin{align*}
b^d Y_m+Z_n &\leq b \max\{Z_n,1\} Y_m + \max\{Z_n,1\} \leq (b+1)\max\{Z_n,1\}Y_m \\&< 2b\alpha^{m+1}\alpha^{m+1} = 2b\alpha^{2m+2}.
\end{align*}
Taking the logarithm of both sides, we obtain:
\begin{align*}
h(b^d Y_m+Z_n) &= \log(b^d Y_m+Z_n) < \log(2b) + (2m+2)\log\alpha \\&= 2m\log\alpha + \log b + \log 2 + 2\log\alpha< 2m\log \alpha + \log b + 3.
\end{align*}

For the denominator, 
$$h(\lambda_X(1-\lambda_T\lambda_X^{-1}\alpha^{r-k}))\leq h(\lambda_X)+h(\lambda_T)+(k-r)h(\alpha)+\log 2\leq \frac{k-r}{2}\log \alpha+3.$$
From Lemma \ref{lem:general_bounds}, by using estimate $b^2+b+1\leq 2b^2$, since $b\geq 2$, we get  
$$h(\lambda_X(1-\lambda_T\lambda_X^{-1}\alpha^{r-k}))\leq m\log \alpha+\log b+6. $$
Hence, 
$$h(\gamma_3)\leq 3m\log \alpha+2\log b+9.$$
Notice that $\gamma_3>0$ and $|\log(\gamma_3)|\leq 2h(\gamma_3)$, hence 
$$A_3:=6m\log\alpha+4\log b+18.$$

Applying Matveev's theorem and using $|\Gamma_{2}|\leq 3/\alpha^k$, we obtain 
\begin{equation}\label{eq:k_pre_bound_1}
k < 4.15  \cdot 10^{25} \cdot (\log b)^2 \cdot (1+\log (3k))^2.
\end{equation}

As noted after Proposition \ref{prop:kje2}, the trivial solution for $k=2$ already satisfies Theorem \ref{tm:main}. Therefore, we can assume $k \geq 3$ and we have
$$
1+\log(3k) < 3\log k.
$$
Substituting this into \eqref{eq:k_pre_bound_1}, we get
$$
\frac{k}{\log^2 k}
<
3.74\cdot10^{26}(\log b)^2.
$$
Let 
$$
H=3.74\cdot10^{26}\log^2 b.
$$
Applying Lemma \ref{lemma:supporting} with $\ell =2$, we obtain
$
k<2^2H(\log H)^2.
$
Therefore
$$
k
<
4\cdot 3.74\cdot10^{26}(\log b)^2
\left(\log(3.74\cdot10^{26})+2\log\log b\right)^2.
$$
Since
$$
\log(3.74\cdot10^{26})<61.2
$$
and 
$$
(61.2+2\log\log b)^2<73^2\log b
\qquad (b\geq2),
$$
we get 
\begin{equation}\label{eq:final_k_1}
k<8\cdot 10^{30}\cdot \log^3 b.
\end{equation}
 By the assumption $n\leq m$ and Lemma~\ref{lem:general_bounds} we have $n\leq m<k$ and $r<k$. Hence the bound \eqref{eq:final_k_1} implies that all the indices are bounded, and therefore there are only finitely many solutions. This completes the proof of Theorem~\ref{tm:main}.

\section{Search for solutions in the range $2\leq b\leq 15$}\label{sec5}

In this section, we transition from the theoretical bounds established in Theorem \ref{tm:main} to explicit numerical results. Before applying the reduction method to completely resolve equation \eqref{eq:main} for bases $b$ such that $2\le b\le 15$ under the assumption $n\leq m$, we briefly present the findings of a computational search  which also includes solutions with $n>m$.

\subsection{Initial Computational Search and Examples}\label{computational_search}

We performed a simple computational search within the range $2\leq b \leq 15$ and indices $m, n, r \le 300$, without any restriction on the relation between $m$ and $n$. The search revealed that there are no solutions to equation \eqref{eq:main} for $b \in \{5, 7, 10, 11\}$ within this parameter range and, under the condition $n \leq m$, yielded the solutions described in Theorem \ref{thm2:main}.

Each of the 16 equations \eqref{eq:main} has at least one solution within these parameters. This remains true even if we require that $P_0 = 0$ is not used in the concatenation; for each of the 16 equations, there exists at least one such solution in some base $2 \le b \le 15$. 

To demonstrate this, we provide one solution for each of the 16 equations:
\begin{align*}
P_{8} &= 6^{1+2} P_1 + 6^2 P_3 + P_4, & P_{8} &= 14^{1+1} P_2 + 14^1 P_1 + Q_1, \\
P_{7} &= 12^{1+1} P_1 + 12^1 Q_0 + P_1, & P_{10} &= 4^{4+1} P_2 + 4^1 Q_5 + Q_1, 
\end{align*}
\begin{align*}
P_{8} &= 13^{1+1} Q_1 + 13^1 P_3 + P_3, & P_{8} &= 14^{1+1} Q_1 + 14^1 P_1 + Q_1, \\
P_{10} &= 4^{4+1} Q_1 + 4^1 Q_5 + P_2, & P_{10} &= 4^{4+1} Q_1 + 4^1 Q_5 + Q_1, 
\end{align*}
\begin{align*}
Q_{3} &= 3^{1+1} P_1 + 3^1 P_1 + P_2, & Q_{8} &= 15^{1+1} P_3 + 15^1 P_1 + Q_3, \\
Q_{7} &= 14^{1+1} P_2 + 14^1 Q_2 + P_2, & Q_{5} &= 8^{1+1} P_1 + 8^1 Q_1 + Q_1, 
\end{align*}
\begin{align*}
Q_{10} &= 4^{1+4} Q_2 + 4^4 P_2 + P_6, & Q_{7} &= 2^{1+4} Q_3 + 2^4 P_1 + Q_3, \\
Q_{10} &= 4^{1+4} Q_2 + 4^4 Q_1 + P_6, & Q_{13} &= 9^{2+2} Q_3 + 9^2 Q_4 + Q_4.
\end{align*}


\begin{remark}
    It is interesting to note that within this parameter range, under the condition $n\leq m$, the equations 
    $$
    P_k = b^{d+l}P_m + b^lQ_n + Q_r \qquad \text{and} \qquad P_k = b^{d+l}Q_m + b^lQ_n + Q_r
    $$
    yield no solutions. Expanding our search to an enlarged base range $2\leq b\leq 1000$, we found that for $b=33$ it holds that
    $$P_{10} = 33^{1+1} P_2 + 33^1 Q_2 + Q_1, $$
    however, the equation $P_k = b^{d+l}Q_m + b^lQ_n + Q_r$ still yields no solutions satisfying $n\leq m$ even within this extended range.
\end{remark}

For each fixed base $b$ with $2 \le b \le 15$, Theorem~\ref{tm:main} provides an explicit, yet computationally infeasible, upper bound on $k$. Since these bounds are far too large to permit a direct search, we must reduce them to a manageable size, ideally matching the range of indices already covered in this subsection. To this end, we apply the reduction method of Lemma~\ref{lemma:reduction}.

\subsection{ The First Reduction}

Let
$$
z_1=(k-m)\log\alpha-(d+l)\log b+\log\left(\dfrac{\lambda_X}{\lambda_Y}\right).  
$$
By the definition of $\Gamma_1$ in \eqref{eq:Gamma1_def}, we have
$\Gamma_1=e^{z_1}-1.$ 
For $m\ge 4$, inequality~\eqref{lem:Gamma1_bound} gives
$$
|\Gamma_1|=|e^{z_1}-1|<8\cdot\alpha^{-m}<\dfrac{1}{2}.
$$
Since $\Gamma_1\neq0$ and $|\log(1+x)|<2|x|$ for $|x|<1/2$, we obtain
$$0<|\log(1+\Gamma_1)|=|z_1|
<2|\Gamma_1|<16\alpha^{-m}.$$
Dividing by $\log b$, we obtain
\begin{equation}\label{eqn:baker_D}
0<\left|(k-m)\dfrac{\log\alpha}{\log b}-(d+l)+\dfrac{\log(\lambda_X/\lambda_Y)}{\log b}\right|<\dfrac{16}{\log b}\cdot\alpha^{-m}.
\end{equation}
This is of the form \eqref{eqn:baker_d} treated in Lemma~\ref{lemma:reduction}, with
$$
\tau=\dfrac{\log\alpha}{\log b},\qquad \mu=\dfrac{\log(\lambda_X/\lambda_Y)}{\log b},\qquad A=\dfrac{16}{\log b},\qquad B=\alpha,\qquad w=m.
$$
For every integer $b\geq2$, the number
$\log\alpha/\log b$ is irrational. Indeed, if the number
$\log\alpha/\log b=p/q$ were rational with positive integers $p,q$,
then $\alpha^q=b^p$. Taking norms in $\mathbb Q(\sqrt2)$ gives
$(-1)^q=b^{2p}$, which is impossible. 

Here the integer $k-m$ plays the role of the variable denoted $m$
in Lemma~\ref{lemma:reduction}. Since $k>m$ by
Lemma~\ref{lem:general_bounds}~a), $k-m$ is a positive integer, and
to apply Lemma~\ref{lemma:reduction} we need an upper bound on it. As $m\ge 1$, we have $k-m<k$, and combining this with the bound on $k$, Theorem~\ref{tm:main} yields
\begin{equation}\label{eq:M-bound}
k-m<k<8\cdot 10^{30}\cdot\log^3 b.
\end{equation}

For each fixed base $b$ with $2\le b\le 15$, we may take $M=\left\lceil 8\cdot10^{30}\log^3 b\right\rceil$, together with the parameters $\tau$, $\mu$, $A$, and $B$ fixed above, to apply Lemma~\ref{lemma:reduction} and reduce this bound to a computationally feasible range. If $X=Y$, then $\mu=0$. In this case,
$\varepsilon=-M\|\tau q\|<0$
for every convergent $p/q$ to $\tau$, since $\tau$ is irrational. Hence
Lemma~\ref{lemma:reduction} cannot be applied, and we treat this homogeneous
case separately using Lemma~\ref{lemma:legendre}.

We now apply the reduction method of Lemma~\ref{lemma:reduction} to the inhomogeneous
inequality~\eqref{eqn:baker_D},  for each base $b$ with $2 \le b \le 15$. For every convergent $p/q$ of the continued fraction expansion of $\tau = \log\alpha/\log b$ with $q > 6M$, we compute
$\varepsilon = \|\mu q\| - M\|\tau q\|$ and retain the first convergent for which
$\varepsilon > 0$. Substituting the corresponding values of $q$ and $\varepsilon$ into
Lemma~\ref{lemma:reduction} then yields, for each $b$, a much smaller upper bound for $m$.  Taking the maximum over all bases $2 \le b \le 15$, which is attained at $b = 15$, we obtain the bound $m \le 91$ for $X\ne Y$ and $2\le b\le15$.

It remains to treat the homogeneous case $\mu = 0$, corresponding to $X = Y$, which was set aside above since inequality~\eqref{eqn:baker_D} then reduces to a purely homogeneous linear form and cannot be handled directly by Lemma~\ref{lemma:reduction}. For this case, we instead apply Lemma~\ref{lemma:legendre}. Then from \eqref{eqn:baker_D} with $b\ge 2$, we have
\begin{align}\label{le}
0<\left|\dfrac{\log\alpha}{\log b}-\dfrac{d+l}{k-m}\right|<\dfrac{16}{\log b}\cdot \dfrac{1}{(k-m)\cdot\alpha^m}<\dfrac{23.0832}{(k-m)\cdot\alpha^m}.
\end{align}
Since $b\le15$, Theorem~\ref{tm:main} gives
$$
k
<8\cdot10^{30}\log^3 b
\le8\cdot10^{30}\log^3 15
<1589\cdot10^{29}.
$$
Thus, we may take
$
M=1589\cdot10^{29},
$
so that $k-m<k<M$.
Assume now that $m > 100$. Then it can be seen that
$$
\frac{\alpha^m}{2(23.0832)} > 4.104\cdot 10^{36} > 1589\cdot 10^{29} > k > k-m,
$$
and then from \eqref{le}, we get
$$
\left|\frac{\log\alpha}{\log b}-\frac{d+l}{k-m}\right| < \frac{23.0832}{(k-m)\cdot\alpha^m} < \frac{1}{2(k-m)^2}.
$$
From Lemma~\ref{lemma:legendre}, we conclude that the rational number $\dfrac{d+l}{k-m}$ is a convergent of the continued fraction expansion of $\kappa := \dfrac{\log\alpha}{\log b}$; that is, $\dfrac{d+l}{k-m} = p_t/q_t$ for some $t$. For each $b$ with $2\le b\le15$, we use Mathematica to find the first
convergent $p_N/q_N$ satisfying $q_N>M$, and set
$a(M):=\max\{a_i:i=0,\dots,N\}$. Lemma~\ref{lemma:legendre} and \eqref{le} then give
$$
\frac{1}{(a(M)+2)\cdot (k-m)^2} \le \left| \frac{\log\alpha}{\log b} - \frac{d+l}{k-m} \right| < \frac{23.0832}{(k-m)\cdot\alpha^m}.
$$
Multiplying by $k-m >0$, this yields
$$
\frac{1}{(a(M)+2)\cdot (k-m)} \le \left| (k-m)\frac{\log\alpha}{\log b} - (d+l) \right| < \frac{23.0832}{\alpha^m}
$$
which leads to
\begin{align}\label{28}
m<\dfrac{\log\left(23.0832\cdot (a(M)+2)\cdot 1589\cdot 10^{29} \right)}{\log\alpha }.
\end{align}

 After computing such $q_N$ and $a(M)$ for each $b$ with $2\leq b\leq 15$, we obtain that $a(M)\leq580$, with the maximum attained at $b=12$. Substituting this bound into \eqref{28}, we obtain
$m < 94.9103$ for all $b$ with $2 \le b \le 15$. This contradicts the assumption that $m > 100$. Combining the two cases, we obtain $m\le100$ for all $2\le b\le15$. 
We return
to the estimate on $k-r$ obtained in Lemma \ref{lem:general_bounds} e), namely
$$
k - r < 2m + 5 + \frac{\log(b^2+b+1)}{\log \alpha}.
$$
Part b) of Lemma \ref{lem:general_bounds} guarantees that $k > r$ for every solution, so
$k-r$ is indeed a well-defined positive integer, to which the above inequality
applies. Since the right-hand side is increasing in both $m$ and $b$,
substituting the bounds $m \le 100$ and $b \le 15$  yields,
uniformly over all bases under consideration,
$$
k - r < 2(100) + 5 + \frac{\log(15^2+15+1)}{\log \alpha}
       = 205 + \frac{\log 241}{\log \alpha} < 205 + 6.224= 211.224,
$$
so that $1\leq k-r\leq211$. Moreover, by
Lemma~\ref{lem:general_bounds} c), if $X=P$, then $k-r\geq2$.
Hence,
\begin{equation}\label{k-r}
\begin{cases}
2\leq k-r\leq211, & \text{if } X=P,\\
1\leq k-r\leq211, & \text{if } X=Q.
\end{cases}
\end{equation}

\subsection{ The Second Reduction }

Now, put
$$
z_2 = l \log b - k \log \alpha
      + \log\!\left(\frac{b^{d}Y_m + Z_n}
      {\lambda_X\bigl(1 - \lambda_T \lambda_X^{-1} \alpha^{r-k}\bigr)}\right)
$$
and we have $\Gamma_2 = 1 - e^{z_2}$.

The case $k=2$ was determined in Proposition~\ref{prop:kje2}, so we may assume
$k\ge3$. By the inequality  $|\Gamma_2|\le 3\alpha^{-k}$, we have
$$
|\Gamma_2|
\le 3\alpha^{-k}
\le 3\alpha^{-3}
<\frac12.
$$

from which it follows that
$
|z_2| < 6 \cdot \alpha^{-k}.
$
Dividing by $\log \alpha$, we obtain
\begin{equation}\label{eqn:baker_D1}
0 < \left| l \, \frac{\log b}{\log \alpha} - k
    + \frac{1}{\log \alpha}
    \log\!\left(\frac{b^{d}Y_m + Z_n}
    {\lambda_X\bigl(1 - \lambda_T \lambda_X^{-1}\alpha^{r-k}\bigr)}\right)
    \right|
< \frac{6}{\log \alpha} \cdot \alpha^{-k}.
\end{equation}
This is again of the form \eqref{eqn:baker_d} treated in Lemma~\ref{lemma:reduction}, with
\begin{align}\label{mu}
\tau = \frac{\log b}{\log \alpha}, \quad
\mu = \frac{1}{\log \alpha}
      \log\!\left(\frac{b^{d}Y_m + Z_n}
      {\lambda_X\bigl(1 - \lambda_T \lambda_X^{-1}\alpha^{r-k}\bigr)}\right),
\quad
A = \frac{6}{\log \alpha}, \; B = \alpha, \; w = k,
\end{align}
where the integer $l$ plays the role of the variable denoted $m$ in
Lemma~\ref{lemma:reduction}. 
Since $l < d + l < 1.28(k+1)$ by Lemma~\ref{lem:general_bounds} d), and since $k$
is bounded by Theorem~\ref{tm:main}, we obtain
$$
l
<1.28\left(8\cdot10^{30}\log^3 15+1\right)
<2034\cdot10^{29}=:M.
$$

Since the parameters in \eqref{eqn:baker_D1} depend on the specific choice of
$X,Y,Z,T$ in \eqref{eq:main}, we perform our analysis separately for each of
the 16 possible equations, and for brevity denote each equation below by its
corresponding ordered quadruple $(X,Y,Z,T)$. For each of them we implemented
the algorithm of Lemma~\ref{lemma:reduction} in Wolfram {Mathematica} and
carried out the corresponding computations for all parameter tuples
$(b,m,n,k-r)$ satisfying
$$
2 \le b \le 15,\qquad 1\le m\le 100,\qquad 0 \le n \le m,\qquad 1 \le k-r \le 211,
$$
where, by Lemma~\ref{lem:general_bounds} c), the range of $k-r$ is further
restricted to $2\le k-r\le211$ whenever $X=P$. For each parameter tuple we run through the successive convergents $p/q$ to
$\tau$ with $q>6M$ until a convergent with $\varepsilon>0$ is reached. It
suffices to consider the first $150$ convergents to $\tau$. All computations were carried out with a working
precision of $250$ significant digits.

The parameter tuples for which $\varepsilon\le0$ holds for every convergent are
exactly those listed in Cases 1--7 below; for all the remaining ones the
resulting upper bounds on $k$ are reported in Table~\ref{tab:kk}.

\begin{table}[htbp]
\centering
\caption{Upper bounds on $k$ for the sixteen choices of $(X,Y,Z,T)$.}
\label{tab:kk}
\begin{tabular}{|c|c||c|c|}
\hline
$(X,Y,Z,T)$ & $k\le$ & $(X,Y,Z,T)$ & $k\le$ \\
\hline\hline
$(P,P,P,P)$ & $207$ & $(Q,P,P,P)$ & $216$ \\ \hline
$(P,P,P,Q)$ & $207$ & $(Q,P,P,Q)$ & $220$ \\ \hline
$(P,P,Q,P)$ & $113$ & $(Q,P,Q,P)$ & $112$ \\ \hline
$(P,P,Q,Q)$ & $115$ & $(Q,P,Q,Q)$ & $110$ \\ \hline
$(P,Q,P,P)$ & $202$ & $(Q,Q,P,P)$ & $215$ \\ \hline
$(P,Q,P,Q)$ & $112$ & $(Q,Q,P,Q)$ & $220$ \\ \hline
$(P,Q,Q,P)$ & $110$ & $(Q,Q,Q,P)$ & $116$ \\ \hline
$(P,Q,Q,Q)$ & $112$ & $(Q,Q,Q,Q)$ & $113$ \\
\hline
\end{tabular}
\end{table}

Thus $k\le220$ for every parameter tuple to which Lemma~\ref{lemma:reduction}
applies. It remains to treat the tuples listed in Cases 1--7. For each of them
the value of $\mu$ simplifies to an expression of the form
$u+v\log b/\log\alpha$ with integers $u,v$. As shown in \eqref{eq:eps_neg}
below, this implies that $\varepsilon\le0$ for every convergent, so that
Lemma~\ref{lemma:reduction} cannot be applied there and must be replaced by
Lemma~\ref{lemma:legendre}.

\medskip

{\bf Case 1.} {\it The case when $(X,Y,Z,T) = (P,P,P,P)$.}

\medskip 

In this case, if $n=0$, $m=2i$ with $i=1,\ldots,50,$  $2\le b\le 15$ and $k-r=2m$, we have $\varepsilon<0$. For these parameters,
$$
\mu = \frac{1}{\log \alpha}
      \log\!\left(\frac{b^{d}P_m + P_n}
      {c\bigl(1 - \alpha^{r-k}\bigr)}\right)=\frac{1}{\log \alpha}
      \log\!\left(\frac{bP_m}
      {c\bigl(1 - \alpha^{r-k}\bigr)}\right).
$$
Since $m$ is even, we have $\beta^m = \alpha^{-m}$, so Binet's formula for the Pell sequence becomes
$$
P_m = \frac{\alpha^m - \beta^m}{2\sqrt{2}} = \frac{\alpha^m - \alpha^{-m}}{2\sqrt{2}}.
$$
Let $x = \alpha^m$. The expression inside the logarithm becomes
$$
\frac{b\,P_m}{c\left(1-\alpha^{-(k-r)}\right)}
= \frac{b\cdot \dfrac{x-\dfrac{1}{x}}{2\sqrt{2}}}
       {\dfrac{1}{2\sqrt{2}}\left(1-\dfrac{1}{x^{2}}\right)}
= \frac{b\left(x-\dfrac{1}{x}\right)}{1-\dfrac{1}{x^{2}}}
= \frac{b\cdot \dfrac{x^{2}-1}{x}}{\dfrac{x^{2}-1}{x^{2}}}
= b\,x.
$$
Substituting back $x=\alpha^m$, we obtain
$$
\frac{b\,P_m}{c\left(1-\alpha^{-(k-r)}\right)} = b\,\alpha^{m},
$$
and therefore 
$$
\mu = m + \frac{\log b}{\log \alpha}.
$$
The remaining quadruples are
\[
\begin{aligned}
(b,m,n,k-r)\in\{&
(2,1,0,4),(2,1,1,8),(2,4,4,16),\\
&(3,1,1,8),(3,5,5,20),(5,2,2,8),\\
&(6,2,0,8),(6,4,0,4),(9,5,5,20),\\
&(11,1,1,8),(12,1,0,8),(13,3,3,12),\\
&(14,3,0,12),(14,5,2,16),(14,7,4,20)\}.
\end{aligned}
\]
For these quadruples, direct simplification gives 
$$
\mu = u + v \frac{\log b}{\log \alpha}
$$
for
$
(u,v)\in
\{(2,0),(2,2),(4,-2),(4,-1),(4,0),
(6,0),(8,-1),(8,0),(10,0)\}.
$

\medskip

{\bf Case 2.} {\it The case when $(X,Y,Z,T) = (P,P,Q,P)$.}

\medskip 

In this case, we obtain $\varepsilon<0$ for
$$
\begin{aligned}
(b,m,n,k-r)\in\{&
(2,1,0,8),(2,1,1,8),(2,4,2,16),\\
&(5,2,0,8),(5,2,1,8),\\
&(10,1,0,8),(10,1,1,8),\\
&(14,5,0,16),(14,5,1,16)\}.
\end{aligned}
$$ For these quadruples, direct simplification gives
\begin{equation}
\mu = u + v \frac{\log b}{\log \alpha},
\qquad
(u,v)\in\{(4,-1),(8,-2),(4,0),(8,0)\}.
\end{equation}

\medskip

{\bf Case 3.} {\it The case when  $(X,Y,Z,T) = (P,Q,P,P)$.}

\medskip

If $n=0$, $m=1$, $2\le b\le 15$, and $k-r=4$, we have $\varepsilon<0$.  For these parameters
$$
\mu = \frac{1}{\log \alpha}
      \log\!\left(\frac{b^{d}Q_m + P_n}
      {c(1 - \alpha^{r-k})}\right)=\frac{1}{\log \alpha}
      \log\!\left(\frac{bQ_1}
      {c(1 - \alpha^{-4})}\right).
$$
Since $c(1 - \alpha^{-4})=2/\alpha^2$ and $Q_1=2$, it follows that
\begin{align}
\mu =\dfrac{\log b}{\log\alpha}+2.
\end{align}
In the remaining cases, we obtain $\varepsilon<0$ for 
$$
\begin{aligned}
(b,m,n,k-r)\in\{&
(2,2,0,8),\ (3,2,0,4),\ (5,3,0,12),\ (6,1,0,8),\\
&(7,3,0,4),\ (12,4,0,16),\ (12,6,2,20),\ (13,7,0,28)\}.
\end{aligned}
$$
For these quadruples, direct simplification of $\mu$ gives
\begin{align}\label{1mu}
\mu = u + v \frac{\log b}{\log \alpha},
\qquad
(u,v)\in\{(4,0),(2,2),(6,0),(8,0),(10,0),(14,-1)\}.
\end{align}

\medskip

{\bf Case 4.} {\it The case when  $(X,Y,Z,T) = (P,Q,Q,P)$.}

\medskip 

In this case, only the quadruples
$$
(b, m, n, k-r) \in \{(4, 2, 2, 16), (5, 1, 0, 8), (5, 1, 1, 8), (12, 6, 0, 20), (12, 6, 1, 20)\}
$$
lead to $\varepsilon < 0$.
For these quadruples, direct simplification of $\mu$ gives
$$
\mu=u+v\frac{\log b}{\log\alpha},
\qquad
(u,v)\in\{(8,-1),(4,0),(10,0)\}.
$$

\medskip

{\bf Case 5.} {\it The case when  $(X,Y,Z,T) = (Q, P, P, Q)$.}

\medskip 
For  $m=2, n=0, k-r=2$ and  $2\le b\le 15,$ we get $\varepsilon < 0.$ The corresponding $\mu$ is $1+\log b/\log\alpha.$ Furthermore, for the quadruples
$$
\begin{aligned}
(b,m,n,k-r)\in\{&
(2,1,0,2),\ (5,6,0,6),\ (6,2,2,6),\ (6,4,0,2),\\
&(7,2,0,6),\ (13,1,1,6),\ (13,14,0,14),\ (14,1,0,6)\}
\end{aligned}
$$
we also get $\varepsilon < 0$. For these quadruples, direct simplification of $\mu$ gives
$$
\mu=u+v\frac{\log b}{\log\alpha},
\qquad
(u,v)\in\{(1,0),(3,2),(3,0),(1,2),(7,3)\}.
$$

\medskip

{\bf Case 6.} {\it The case when  $(X,Y,Z,T) = (Q,Q,P,Q)$.}

\medskip 

In this case, if $n=0$, $2\le b\le 15,$  $m=2i+1$ with $i=0,\ldots,49$  and $k-r=2m$, we have $\varepsilon<0$.  For these parameters,
$$
\mu = \frac{1}{\log \alpha}
      \log\!\left(\frac{b^{d}Q_m + P_n}
      {1 - \alpha^{r-k}}\right)=\frac{1}{\log \alpha}
      \log\!\left(\frac{bQ_m}
      {1 - \alpha^{r-k}}\right).
$$
Since $m = 2i+1$ is odd, we have $\beta^{m} = (-\alpha^{-1})^{m} = -\alpha^{-m}$, hence
$
Q_m = \alpha^{m} - \alpha^{-m}.$ With $k-r= 2m$, we have
$$
1 - \alpha^{-(k-r)} = 1 - \alpha^{-2m} = \alpha^{-m}\left(\alpha^{m} - \alpha^{-m}\right) = \alpha^{-m} Q_m.
$$
Therefore
$$
\frac{b\,Q_m}{1-\alpha^{-(k-r)}} = \frac{b\,Q_m}{\alpha^{-m}Q_m} = b\,\alpha^{m},
$$
which leads to
\begin{align}
\mu = \frac{1}{\log \alpha}\log\left(b\,\alpha^{m}\right) = \frac{\log b}{\log \alpha} + m.
\end{align}
Also, in case of the quadruples 
$$
(b, m, n, k-r) \in \{(3, 2, 0, 2), (7, 1, 0, 6), (7, 3, 0, 2), (14, 4, 2,14)\}
$$
we have $\varepsilon<0$, and direct simplification gives
\begin{align}
\mu = u + v \frac{\log b}{\log \alpha},\quad 
(u,v)\in\{(1,2),(3,0),(7,0)\}.
\end{align}

\medskip

{\bf Case 7.} {\it The case when  $(X,Y,Z,T)\in \{(Q, P, Q, Q), (Q, Q, Q, Q)\}$.}

\medskip 
In this case we have $\varepsilon < 0$ only for the following quadruples
$$
\begin{aligned}
(b, m, n, k-r) \in \{&(2, 5, 3, 14),\ (4, 5, 3, 14),\ (6, 2, 1, 6),\\
&(6, 2, 0, 6),\ (12, 1, 0, 6),\ (12, 1, 1, 6)\}
\end{aligned}
$$
in case of $(X, Y, Z, T) = (Q, P, Q, Q)$, and
$$
(b, m, n, k-r) \in \{(6, 1, 1, 6),\ (6, 1, 0, 6),\ (14, 4, 0, 14),\ (14, 4, 1, 14)\}
$$
in case of $(X, Y, Z, T) = (Q, Q, Q, Q)$. For all the quadruples listed in this case, direct simplification of the corresponding expression for $\mu$ gives
$\mu\in\{3,7\}.$

Combining Cases 1--7, we can restate everything in terms of 
\begin{align}\label{3mu}
\mu = u + v \frac{\log b}{\log \alpha},
\end{align}
for some integers $u$ and $v$ satisfying
$$
1\le u\le100,
\qquad
v\in\{-2,-1,0,1,2,3\}.
$$
Thus, \eqref{eqn:baker_D1} becomes
\begin{equation}\label{1Legendreb=15}
0 < \left|(l+v)\, \tau - (k-u)\right|
< \frac{6}{\log \alpha} \cdot \alpha^{-k}.
\end{equation}
Note that for these parameter tuples $\varepsilon> 0$ never holds. Indeed, since $u\in\mathbb Z$
and $|v|\le3$, for every convergent $p/q$ to $\tau$ we have
$\|\mu q\|=\|v\tau q\|\le|v|\cdot\|\tau q\|\le3\|\tau q\|$, and therefore
\begin{equation}\label{eq:eps_neg}
\varepsilon=\|\mu q\|-M\|\tau q\|\le(3-M)\|\tau q\|<0 .
\end{equation}
This is why we now turn to Lemma~\ref{lemma:legendre}.

Assume first that $l+v\ge1$ and $k>100$. Since $l<M$ and $v\le3$, we have
$l+v<M+3$. Moreover,
$$
\dfrac{\alpha^k}{12/\log \alpha}
>1.3917\cdot 10^{37}
>M+3
>l+v\ge1,
$$
and therefore, from \eqref{1Legendreb=15}, we obtain
$$
\left|\frac{\log b}{\log \alpha} - \dfrac{k-u}{l+v}\right|
< \frac{6}{\log \alpha} \cdot\dfrac{1}{(l+v)\alpha^k}
<\dfrac{1}{2(l+v)^2}.
$$
It follows from Lemma~\ref{lemma:legendre} that the rational number
$\dfrac{k-u}{l+v}$ is a convergent to
$\kappa :=\tau= \dfrac{\log b}{\log \alpha}$. For each $b$ with $2\le b\le15$, let $p_N/q_N$ be the first convergent to $\kappa$ whose denominator satisfies $q_N>M+3$, and put
$
a(M+3):=\max\{a_i:i=0,\ldots,N\}.
$ The computation gives
$
a(M+3)\le580.
$
Therefore, by Lemma~\ref{lemma:legendre},
\begin{align*}
\dfrac{1}{(a(M+3)+2)(l+v)} 
\le
\left|(l+v)\frac{\log b}{\log \alpha}-(k-u)\right|
<
\frac{6}{\log \alpha}\cdot\alpha^{-k}.
\end{align*}
This leads to
$$
k<
\frac{
\log\left(
(6/\log\alpha)(a(M+3)+2)(l+v)
\right)
}{
\log\alpha
}.
$$
Since $a(M+3)\le580$ and $l+v<M+3$, we obtain
$$
k<93.8050,
$$
which contradicts $k>100$.

Therefore,  if  $l+v\ge1$ we conclude that  $k\le 100$.

It remains to consider the case $l+v<1$. Since $l\ge1$ and $v\ge-2$,
this implies $l\in\{1,2\}$. If $r=0$, then clearly $r\le8$.
If $r\ge1$, combining \eqref{eq:seq_bounds} and \eqref{eq:digit_bounds},
we obtain
$$
r<2+l\frac{\log b}{\log\alpha}
\le2+2\frac{\log15}{\log\alpha}
<8.1451.
$$
Hence $r\le8$. Using \eqref{k-r}, we obtain
$
k\le219.
$

Combining these bounds with those in Table~\ref{tab:kk}, we obtain $k\le220$ for all sixteen equations. From Lemma \ref{lem:general_bounds} we also have $n\leq m<k\leq 220$ and $r<k\leq 220.$

\subsection{Final search}
Since all sequence indices are now explicitly bounded under the condition $n \le m$, we confirm that the solutions are exactly the corresponding subset of those found in the search from Subsection \ref{computational_search}, and that bases $b \in \{5, 7, 10, 11\}$ indeed admit no solutions for $n \le m$. This completes the proof of Theorem~\ref{thm2:main}. 

\section*{Statements and Declarations}

\noindent\textbf{Funding.}
The first author was supported by IMSP, Institut de Math\'ematiques et de Sciences Physiques, Universit\'e d'Abomey-Calavi. The second and third authors were supported by the University of Split, grant no.\@ IP-UNIST-44, funded by the European Union – NextGenerationEU.

\medskip
\noindent\textbf{Competing interests.}
The authors have no interests to disclose.

\medskip
\noindent\textbf{Code availability.}
The Mathematica code used for the computations in this study is available from the corresponding author on request.

\medskip
\noindent\textbf{Data availability.}
Data sharing is not applicable to this article as no datasets were generated or analysed during the current study.

\medskip
\noindent\textbf{Declaration on the use of AI.} 
The authors used ChatGPT (OpenAI) and Claude (Anthropic) for English-language proofreading. They reviewed all suggested changes and take full responsibility for the final manuscript.

\end{document}